\documentclass[a4paper,reqno,
dvipdfmx,  
11pt]{amsart}
\usepackage{
amssymb,
amsmath,
amsthm,
eucal,
empheq,
cases,
dsfont,
multicol,
mathrsfs,
tikz,
graphicx,
hyperref
}

\usepackage{lipsum}
\makeatletter
\renewcommand*{\eqref}[1]{%
\hyperref[{#1}]{\textup{\tagform@{\!\!\ref*{#1}}}}%
}

\makeatletter
 
  \@addtoreset{equation}{section}
 \makeatother
\makeatletter

\newcommand{\opnorm}{\@ifstar\@opnorms\@opnorm}
\newcommand{\@opnorms}[1]{%
  \left|\mkern-1.5mu\left|\mkern-1.5mu\left|
   #1
  \right|\mkern-1.5mu\right|\mkern-1.5mu\right|
}
\newcommand{\@opnorm}[2][]{%
  \mathopen{#1|\mkern-1.5mu#1|\mkern-1.5mu#1|}
  #2
  \mathclose{#1|\mkern-1.5mu#1|\mkern-1.5mu#1|}
}

\makeatother\theoremstyle{plain}

\newtheorem{theorem}{Theorem}[section]
\newtheorem{lemma}[theorem]{Lemma}
\newtheorem{proposition}[theorem]{Proposition}

\theoremstyle{definition}

\newtheorem{remark}[theorem]{Remark}

\newtheorem{assumption}{Assumption}

\def\supp{\mathop{\mathrm{supp}}\nolimits}

\def\Re{\mathop{\mathrm{Re}}\nolimits}

\def\R{{\mathbb{R}}}
\def\Z{{\mathbb{Z}}}

\def\C{{\mathbb{C}}}

\def\F{{\mathcal{F}}}

\newcommand{\J}[1]{\left\langle #1 \right\rangle}
\def\<{{\langle}}
\def\>{{\rangle}}

\def\ep{{\varepsilon}}
\def\ds{\displaystyle}

\title[Modified scattering for the 1D cubic NLS with long-range potential]{Modified scattering for the cubic nonlinear Schr\"odinger equation with long-range potentials in one space dimension}

\author{Masaki Kawamoto}
\address[M. Kawamoto]{Research Institute for Interdisciplinary Science, Okayama University, 3-1-1, Tsushimanaka, Kita-ku, Okayama City, Okayama, 700-8530, Japan}
\email{kawamoto.masaki@okayama-u.ac.jp}
\author{Haruya Mizutani}
\address[H. Mizutani]{Department of Mathematics, Graduate School of Science, Osaka University, Toyonaka, Osaka 560-0043, Japan}
\email{haruya@math.sci.osaka-u.ac.jp}

\makeatletter
\@namedef{subjclassname@2020}{%
	\textup{2020} Mathematics Subject Classification}
\makeatother


\begin{document}

%
\begin{abstract}We consider the cubic nonlinear Schr\"odinger equation with long-range linear potentials in one space dimension, and prove the modified scattering in the energy space for the associated final state problem with a prescribed small asymptotic profile. Compared with the leading term of the free solution, the asymptotic profile has an additional phase correction depending both on the long-range part of the potential and on the nonlinear term. The proof is based on a simple energy method and does not rely on global-in-time Strichartz estimates for Schr\"odinger equations with linear potentials. In particular, the class of potentials to which our theorem applies is large enough to accommodate slowly decaying negative potentials so that the associated Schr\"odinger operators may have negative eigenvalues. 
\end{abstract}

\maketitle

\section{Introduction}
The present paper is concerned with the scattering problem for the following cubic nonlinear Schr\"odinger equation (NLS) with a linear potential  in one space dimension: 
\begin{align}
\label{NLS_V}
i\partial_t u-H_0u-Vu=F(u),\quad x\in \R,\quad t\in \R,
\end{align}
where $u=u(t,x)$ is a $\C$-valued unknown function, $
F(u)=\lambda |u|^{2}u
$ with $\lambda \in \R$, and $H_0+V$ is the Schr\"odinger operator with a real-valued potential $V$: 
$$
H_0:=-\frac12\frac{d^2}{dx^2},\quad V=V(x):\R\to \R. 
$$
We are interested in the (small data) modified scattering for the associated finial state problem. More precisely, given a prescribed sufficiently small asymptotic profile function $u_{\mathrm p}(t,x)$, we construct a unique global solution $u\in C(\R;H^1(\R))$ to \eqref{NLS_V} which scatters to $u_{\mathrm p}$ as $t\to \infty$:  
$$
\lim_{t\to \infty}\|u(t)-u_{\mathrm p}(t)\|_{H^1(\R)}=0, 
$$
where $H^s(\R)$ denotes the $L^2$-Sobolev space of  order $s$. Here the term ``modified" means that the asymptotic profile $u_{\mathrm p}$ has an additional correction term compared with the leading term of the free solution $e^{-itH_0}u_+$, depending both on $F(u)$ and on the long-range part of $V$ (see \eqref{v_p} for the definition of $u_{\mathrm p}$). This gives the first positive result on the modified scattering for \eqref{NLS_V}, and is a continuation of our resent work \cite{Kawamoto_Mizutani_TAMS} for the two and three space dimensional cases. 

To state the result, we first introduce the class of potentials. 

\begin{assumption}
\label{assumption_A}
$V$ is decomposed into the short-range, long-range and compactly supported singular parts as $V=V^{\mathrm S}+V^{\mathrm L}+V^{\mathrm{C}}
$ with real-valued functions $V^{\mathrm S}, V^{\mathrm L},V^{\mathrm{C}}$ satisfying following:
\begin{itemize}
\item Short-range part: $V^{\mathrm S}\in C^1(\R)$ and there exists $\rho_{\mathrm S}>3/2$ such that, for any $k=0,1$,  
\begin{align*}
\left|\frac{d^k}{dx^k} V^{\mathrm S}(x) \right|\lesssim \<x\>^{-\rho_{\mathrm S}} 
\end{align*} 
\item Long-range part: $V^{\mathrm L}\in {C^3} (\R)$ and there exists $\rho_{\mathrm L}>1/2$ such that, for any $k=0,1,2, 3$, 
\begin{align*}
\left |\frac{d^k}{dx^k}V^{\mathrm L}(x)\right|\lesssim \<x\>^{-\rho_{\mathrm L}-k}.
\end{align*}
\item Compactly supported singular part: $V^{\mathrm{C}}\in L^2(\R)$ and $V^{\mathrm{C}}$ is compactly supported. 
\end{itemize}
\end{assumption}
Under Assumption \ref{assumption_A}, $V$ is infinitesimally $H_0$-bounded, that is, for any $\ep>0$ there exists $C_\ep>0$ such that 
$$
\|Vf\|_{L^2}\le \ep \|H_0 f\|_{L^2}+C_\ep \|f\|_{L^2},\quad f\in H^2(\R).
$$
In particular, $H:=H_0+V$ with domain $D(H)=H^2(\R)$ is self-adjoint on $L^2(\R)$ by the Kato--Rellich theorem, and generates the associated unitary group $e^{-itH}$ on $L^2(\R)$. A typical example of $V$ we have in mind is of the form 
$$
V(x)=Z\<x\>^{-\rho}+V^{\mathrm{C}}(x)
$$ with some $\rho>1/2$ and $Z\neq0$. When $Z<0$, it is well known that $H$ has negative eigenvalues.

We next introduce the asymptotic profile $u_{\mathrm p}$. Let $c_0>0$ and $\chi\in C_0^\infty(\R)$ be such that $0\le \chi \le1$, $\chi(x)=1$ for $|x|\le c_0/4$ and $\chi(x)=0$ for $|x|\ge {c_0/3}$. Define an effective potential 
\begin{align}
\label{V_T_1}
V_{T_1}(t,x)=V^{\mathrm{L}}(x)\left\{1-\chi\left(\frac{2x}{t+T_1}\right)\right\},
\end{align}
where $T_1\ge1$ is a large constant specified in Proposition \ref{proposition_HJ} below. It is easily seen that
\begin{itemize}
\item $V_{T_1}\equiv V^{\mathrm {L}}$ if $t\ge0$ and $|x|\ge c_0(t+T_1)/4$;
\item $V_{T_1}\in C^\infty([0,\infty);C^3(\R))$ and $|\partial_x^k V_{T_1}(t,x)|\le C_\alpha\<t\>^{-\rho_{\mathrm{L}}-k}$ on $[0,\infty)\times \R$. 
\end{itemize}

\begin{proposition}
\label{proposition_HJ}
Let $c_0>0$. Then, for sufficiently large $T_1>0$, one can construct a global solution $\Psi\in C^1([1,\infty)\times\R)\cap C([1,\infty); C^3(\R))$ to the Hamilton--Jacobi equation
\begin{align}
\label{HJ}
-\partial_t \Psi(t,x)=\frac12|\partial_x \Psi(t,x)|^2+V_{T_1}(t,x)
\end{align}
such that, for $k=1,2,3$ and $t\ge1$, 
\begin{align}
\label{proposition_HJ_1}
\left\|\partial_x^k\left(\Psi(t,x)-\frac{x^2}{2t}\right)\right\|_{L^\infty(\R)}\lesssim 
\begin{cases}
 \<t\>^{1-\rho_{\mathrm L}- k} & \mbox{if} \ \rho_{\mathrm{L}}  \neq 1, \\ 
\J{t}^{-k}  \log (tT_1^{-1} +1) & \mbox{if} \ \rho_{\mathrm{L}} =1.
 \end{cases}
\end{align}
In particular, for any $\rho_{\mathrm L}'<\rho_{\mathrm L}$, 
$$
\left\|\partial_x^k\left(\Psi(t,x)-\frac{x^2}{2t}\right)\right\|_{L^\infty(\R)}\lesssim  \<t\>^{1-\rho_{\mathrm L}'- k},\quad t\ge1, \quad k=1,2,3. 
$$
\end{proposition}

\begin{proof}
The proposition is basically well-known in the context of long-range linear scattering theory (see e.g. \cite[Sections 1.5, 1.8 and A.3]{DeGe}). Moreover, the proof is essentially same as that of \cite[Proposition 3.1]{Kawamoto_Mizutani_TAMS}. We thus give an outline of the proof in Appendix \ref{appendix_proof_proposition_HJ} below, and refer to \cite[Proposition 3.1]{Kawamoto_Mizutani_TAMS} for more details. 
\end{proof}

Using the function $\Psi(t,x)$, we define $u_{\mathrm p}$ by
\begin{align}
\label{v_p}
u_{\mathrm p}(t,x)=[\mathcal M_\Psi(t)\mathcal D(t)w_{\mathrm p}](t,x)=(it)^{-1/2}e^{i\Psi(t,x)}e^{-i\lambda|\widehat{u_+}(x/t)|^2\log |t|}\widehat{u_+}(x/t),
\end{align}
where $u_+(x)$ is a given function (often called the scattering datum) and 
\begin{align*}
w_{\mathrm{p}}(t,x)&=e^{-i\lambda  |\widehat{u_+}(x)|^{2}\log |t|}\widehat{u_+}(x),\\
\mathcal D(t)f(x)&=(it)^{-1/2}f(x/t),\\
\mathcal M_\Psi(t)f(x)&=e^{i\Psi(t,x)}f(x),\\
\widehat f(x)=\F f(x)&=\frac{1}{\sqrt{2\pi}}\int_\R e^{-ix\cdot\xi}f(\xi)d\xi. 
\end{align*}
We here recall that the leading term of the free solution $e^{-itH_0}u_+$ satisfies
$$
e^{-itH_0}u_+=\mathcal M(t)\mathcal D(t)\widehat{u_+}+o(1)
$$
in $L^2$ as $t\to \infty$, where $\mathcal M(t)f(x)=e^{i|x|^2/(2t)}f(x)$ and the phase function $\Psi_0=|x|^2/(2t)$ solves the free Hamilton--Jacobi equation $-\partial_t\Psi_0=\frac12|\partial_x\Psi_0|^2$. Thus, compared with this free profile $\mathcal M(t)\mathcal D(t)\widehat{u_+}$,  $u_{\mathrm p}$ has the additional phase correction terms $e^{-i\lambda  |\widehat{u_+}(x/t)|^{2}\log |t|}$ and $e^{i(\Psi-\Psi_0)}$ depending on $F(u)$ and $V^{\mathrm L}$, respectively. 

Let $H^{s,r}(\R)$ be the weighted $L^2$-Sobolev space defined by
\begin{align*}
H^{s,r}(\R)&=\{f\in \mathcal S'(\R)\ |\ \|f\|_{H^{s,r}}<\infty\},\\ \|f\|_{H^{s,r}}&=\|\<\partial_x\>^s\<x\>^rf\|_{L^2(\R)}=\|\<\xi\>^s\F[\<x\>^rf]\|_{L^2(\R)},
\end{align*}
where $\<x\>=\sqrt{1+|x|^2}$. Note that $H^{s}(\R)=H^{s,0}(\R)$. We now state the result.

\begin{theorem}
\label{theorem_NLS_V_1}
Suppose that $V$ satisfies Assumption \ref{assumption_A}. Let $c_0>0$ and $T_1$ in Proposition \ref{proposition_HJ} be fixed. Let $b>2$ and  $\delta\le1$ be such that 
$$
\frac12<\delta<\min\{\rho_{\mathrm L},\rho_{\mathrm S}-1\}.
$$
Then, for any $u_+\in H^{1,2\delta}(\R)$ with sufficiently small $\|\widehat{u_+}\|_{L^\infty}$ and $\supp\widehat{u_+}\subset \{|\xi|\ge c_0\}$, there exists a unique solution $u\in C(\R;H^1(\R))$ to \eqref{NLS_V} satisfying the prescribed asymptotic condition:
\begin{align}
\label{theorem_NLS_V_1_1}
\|u(t)-u_{\mathrm p}(t)\|_{H^1(\R)}\lesssim t^{-\delta}(\log t)^b,\quad t\to \infty. 
\end{align}
\end{theorem}

\begin{remark}
Since $\|u_{\mathrm p}(t)\|_{L^\infty}\le |t|^{-1/2}\|\widehat{u_+}\|_{L^\infty}$, the solution $u$ also enjoys the same $L^\infty$-decay estimate as for the free solution $e^{-itH_0}u_+$ by the Sobolev embedding: 
$$
\|u(t)\|_{L^\infty(\R)}\lesssim t^{-1/2},\quad t\to \infty. 
$$
\end{remark}

The modified scattering for the standard NLS
\begin{align}
\label{NLS}
i\partial_t u-H_0u=\lambda |u|^{2/d}u,\quad x\in \R^d,\quad t\in \R,\quad \lambda\in \R,\quad d=1,2,3,
\end{align}
 has been extensively studied both for the final state and  Cauchy problems. It is impossible to list all of the known literature, instead we refer to some of important results \cite{Ozawa_1991, Ginibre_Ozawa_1993, Ginibre_Velo_2001, Carles_2001, Hayashi_Naumkin_2006, HaWaNa} for the final state problem and \cite{HaNa1998,Deift_Zhou_2003, LiSo,KaPu,IfTa} for the Cauchy problem. We also refer to \cite{MaMi,MMU,Iflim_Tataru_2023} for more recent development. For the cubic NLS \eqref{NLS_V} satisfying the very short-range condition $\<x\>V\in L^1(\R)$, there also many works on the modified scattering (see e.g. \cite{IPNa,INa,GePuRo,ChePu,ChePu2,Se,MaMuSe}), where the correction term of the asymptotic profile compared with the free solution is independent of the potential. However, if both of the linear potential $V$ and nonlinear term $F(u)$ are of long-range type, then the literature is much more sparse. In fact, although it has been shown in \cite{Murphy_Nakanishi_2011} that no non-trivial solutions scatter to the free solution, there was no previous positive result before \cite{Kawamoto_Mizutani_TAMS} in such a mixed situation. 

In the previous paper \cite{Kawamoto_Mizutani_TAMS}, we established the modified scattering for the following NLS:  
\begin{align}
\label{NLS_V_2}
i\partial_t u-H_0u-Vu=\lambda |u|^{2/d}u,\quad x\in \R^d,\quad t\in \R,\quad d=2,3.
\end{align}
where $V$ is of long-range type and assumed to satisfy not only a similar condition as Assumption \ref{assumption_A}, but also the repulsive condition, which roughly means that $V$ is positive and monotonously decreasing in the radial direction. 
The proof essentially relied on the global-in-time Strichartz estimates for $e^{-itH}$ proved by \cite{Mizutani_JFA,Taira}. Very recently, the argument by \cite{Kawamoto_Mizutani_TAMS} has been applied to the case with a long-range Hartree-type nonlinearity by \cite{Huang}. Although the strategy of \cite{Kawamoto_Mizutani_TAMS} also works well in one space dimension, the validity of global-in-time Strichartz estimates for $e^{-itH}$ with long-range potentials is completely open for $d=1$ even if $V$ is repulsive. Moreover, if $H$ has an eigenvalue then such global-in-time estimates cannot hold, and one should replace $e^{-itH}$ by its absolutely continuous part $e^{-itH}P_{\mathrm{ac}}(H)$ with the projection $P_{\mathrm{ac}}(H)$ onto the absolutely continuous spectral subspace of $H$. However, our previous argument does not work with $e^{-itH}P_{\mathrm{ac}}(H)$ instead of $e^{-itH}$. Moreover, the validity of global-in-time Strichartz estimates for $e^{-itH}P_{\mathrm{ac}}(H)$ with long-range (non-repulsive) potentials is also completely open. 

Compared with the previous result \cite{Kawamoto_Mizutani_TAMS}, the main new feature of the present paper is to avoid the use of global-in-time Strichartz estimates for $e^{-itH}$, instead the proof is based on a rather simple energy method. This is not only a technical issue, but also crucial to deal with a wider class of potentials than that of \cite{Kawamoto_Mizutani_TAMS}. Indeed, as already explained, our theorem applies to slowly decaying negative potentials $V$ so that $H$ may have (infinitely many) negative eigenvalues. Such a situation naturally appears in several important models in mathematical physics, such as the many-body Schr\"odinger equations having a ground state for which \eqref{NLS_V} can be regarded as a reduced model in the framework of the mean field approximation. It is worth mentioning that our result is not contradict with the fact (see e.g. \cite{Mizumachi}) that \eqref{NLS_V} could have ground state solutions in the case when $H$ has negative eigenvalues or the nonlinear term $F(u)$ is focusing (i.e. $\lambda<0$) since Theorem \ref{theorem_NLS_V_1} only provides the existence of modified scattering solutions, and does not exclude the possible existence of such non-decaying solutions.

\begin{remark}
(1) The assumption $\supp\widehat{u_+}\subset \{|\xi|\ge c_0\}$ with some $c_0>$ is mainly used to prove Proposition \ref{proposition_HJ}. 
We expect that this is a technical condition and should be removed since it is not needed both for the purely linear case $\lambda=0$  and for the purely nonlinear case $V\equiv0$. This is mainly because that, in the linear case,  one can use the density argument and an approximate solution to \eqref{HJ} with $V_{T_1}$ replaced by $V^{\mathrm L}$ to assume without loss of generality that $\widehat{u_+}\in C_0^\infty(\R\setminus\{0\})$. On the other hand, such a reduction is impossible for \eqref{NLS_V} due to the presence of the phase correction $e^{-i\lambda|\widehat{u_+}(x/t)|^2\log |t|}$ in  the formula of $u_{\mathrm p}$ since for two difference scattering data $u_+$ and $v_+$, the difference $e^{-i\lambda|\widehat{u_+}(x/t)|^2\log |t|}-e^{-i\lambda|\widehat{v_+}(x/t)|^2\log |t|}$ may glow as $t\to \infty$. \\
(2) It would be interesting whether a similar result as Theorem \ref{theorem_NLS_V_1} also holds for \eqref{NLS_V_2} with $d=2,3$ and non-repulsive potentials. The main obstruction is that one has to work in the fractional Sobolev space $H^s(\R^d)$ with $d/2<s<1+d/2$ due to the low regularity of $F(u)$. In our argument, we use an explicit formula of $\partial_x\mathcal M_{\Psi}(t)\mathcal D(t)$ (see Lemma \ref{lemma_NLS_V_1} below), while it is difficult to obtain explicit formulas of non-local operators $|\nabla_x|^s\mathcal M_{\Psi}(t)\mathcal D(t)$ and $\<\nabla_x\>^s\mathcal M_{\Psi}(t)\mathcal D(t)$. For the standard NLS \eqref{NLS}, as did in \cite{Hayashi_Naumkin_2006}, one can overcome this difficulty to work with  the energy norm $\|\mathcal D(t)^{-1}\mathcal M(t)^{-1}f\|_{H^s}$ instead of $\|f\|_{H^s}$, or equivalently, to deal with $\mathcal D(t)^{-1}\mathcal M(t)^{-1}u$ in $H^s$ instead of $u$ in $H^s$. This is possible since, by virtue of the explicit formula of $e^{-itH_0}$ (see \eqref{Dollard}), one has $\mathcal D(t)^{-1}\mathcal M(t)^{-1}e^{-itH_0}\mathcal M(t)\mathcal D(t)=e^{iH_0/t}$ and $\|e^{iH_0/t}f\|_{H^s}=\|f\|_{H^s}$. It seems to be  however impossible to obtain such nice properties for $\mathcal D(t)^{-1}\mathcal M_{\Psi}(t)^{-1}e^{-itH}\mathcal M_{\Psi}(t)\mathcal D(t)$. 
\end{remark}

\noindent 
\textbf{Organization of the paper.} We prepare some preliminary lemmas in Section \ref{section_2}. The proof of Theorem \ref{theorem_NLS_V_1} is given in Section \ref{section_3}. Appendix \ref{appendix_proof_proposition_HJ} is devoted to the proof of Proposition \ref{proposition_HJ}. 

\section{Preliminary materials}
\label{section_2}
We here prepare several basic facts used in the paper. In what follows we use the notation
$$
\|f\|=\|f\|_{L^2(\R)},\quad \|f\|_s=\|f\|_{H^s(\R)}. 
$$
Let $U_0(t)=e^{-itH_0}$ be the free Schr\"odinger group, which satisfies the Dollard decomposition
\begin{align}
\label{Dollard}
U_0(t)\mathcal F^{-1}=\mathcal M(t)\mathcal D(t)\mathcal F\mathcal M(t)\mathcal F^{-1}=\mathcal M(t)\mathcal D(t)+\mathcal M(t)\mathcal D(t)\mathcal R(t),
\end{align}
where $\mathcal D(t)$ and $\mathcal M(t)$ have been defined in the introduction,  and 
$$
\mathcal R(t)=\F(\mathcal M(t)-1)\F^{-1}. 
$$

\begin{lemma}
\label{lemma_R}
Let $1\le p\le \infty$. Then
\begin{align}
\label{lemma_D}
\|\mathcal D(t)f\|_{L^p}=|t|^{-1/2+1/p}\|f\|_{L^p},\quad t\neq0.
\end{align}
Moreover, for $0\le \delta\le1$, $s\in \R$ and $t\neq0$, 
\begin{align}
\label{lemma_R_1}
\|\mathcal R(t) f\|_s&\lesssim |t|^{-\delta}\|f\|_{s+2\delta},\\
\label{lemma_R_2}
\| x R(t) f\|&\lesssim |t|^{-\delta}\|xf\|_{2\delta}+|t|^{-1}\|f\|_{1},
\end{align}
where the implicit constants independent of $t$. 
\end{lemma}

\begin{proof}
 \eqref{lemma_D} follows by an elementary calculation. To prove \eqref{lemma_R_1}, it is enough to observe
$$
|e^{i|x|^2/(2t)}-1|=2\left|\sin \frac{x^2}{4t}\right|\lesssim \left|\frac{x^2}{t}\right|^\delta
$$ for $0\le \delta\le1$, and $\mathcal R(t)$ commutes with $\<\partial_x\>^s$. Finally, we have
$$
x\mathcal R(t)=\mathcal F\{(\mathcal M-1)(-i\partial_\xi)+ t^{-1}\mathcal M\xi\}\mathcal F^{-1}=\mathcal Rx+it^{-1}\mathcal F\mathcal M\mathcal F^{-1}\partial_x
$$
and \eqref{lemma_R_2} follows. 
\end{proof}

\begin{lemma}
\label{lemma_F}
For all $z_0,z_1\in \C$, 
$$
F(z_1)=F(z_0)+2\lambda|z_0|^{2}(z_1-z_0)+\lambda z_0^2\overline{(z_1-z_0)}+G(z_1-z_0,z_0), 
$$
where $G(z,z_0)=2\lambda\Re[zz_0]z+\lambda|z|^2z_0+\lambda|z|^2z$. 
\end{lemma}

\begin{proof}
The lemma follows by a direct calculations or applying the Taylor formula
$$
f(z_1)=f(z_0)+(z_1-z_0)\int_0^1\partial_zf(z_\theta)d\theta+\overline{(z_1-z_0)}\int_0^1\partial_{\overline z}f(z_\theta)d\theta, 
$$
to $f(z)=\lambda|z|^2z$, where $z_\theta=z_0+\theta(z_1-z_0)$. 
\end{proof}

To estimate the $H^s$-norm of the terms $w_{\mathrm{p}}$ and $F(w_{\mathrm{p}})$  defined in \eqref{v_p}, we use the following

\begin{lemma}
\label{lemma_nonlinear}
Let $1/2<s\le3$. Then, for all $t\ge2$, 
\begin{align*}
\|w_{\mathrm{p}}(t)\|_{s}\lesssim \<\log t\>^{\lceil s\rceil}(1+\|\widehat{u_+}\|_{s}^{1+2\lceil s\rceil}),\quad
\|F(w_{\mathrm{p}}(t))\|_{s}\lesssim \<\log t\>^{\lceil s\rceil}(1+\|\widehat{u_+}\|_{s}^{3+2\lceil s\rceil}),
\end{align*}
where $\lceil s\rceil=\min\{m\in \Z\ |\ m\ge s\}$ denotes the smallest integer greater than or equal to $s$. 
\end{lemma}

\begin{proof}
The case $1/2<s<2$ has been proved by \cite[Lemma 4]{Hayashi_Naumkin_2006} (see also \cite[Lemma 2.2]{Kawamoto_Mizutani_TAMS} and its proof). For $s=2,3$, the desired bound can be verified easily by calculating $\partial_x^sw_{\mathrm p}$ and $\partial_x^s F(w_{\mathrm p})$. 
Suppose $2< s <3$ and define $I := i \bar{u} (\partial_x u)  + i  u \overline{ (\partial _x u)}  $. Then
\begin{align*}
e^{-i |u|^{2} }\partial _x^2 ( e^{i |u|^{2} } u )  = I^2 u + 2 I (\partial_x u) + (\partial_x^2 u)  + (\partial _x I) u .
\end{align*} 
Hence for $0 < s' = s-2< 1$, it is enough to deal with 
\begin{align*}
\| e^{i |u|^{2} } I^2 u \|_{s'} , \quad \| e^{i |u|^{2} } I (\partial _x u)  \| _{s'}, \quad \| e^{i |u|^{2} } (\partial_x I) u \| _{s'}. 
\end{align*} 
By \cite[Lemma B.1]{Kawamoto_Mizutani_TAMS}, the above three terms can be estimated as 
\begin{align*}
(1 + \| u \|_s^{2} )\| I^2 u \|_{s'} , \quad (1 + \| u \|_s^{2} )\|  I (\partial_x u)  \| _{s'}, \quad  (1 + \| u \|_s^{2} )\|  (\partial_x I) u \| _{s'}. 
\end{align*} 
By the sharp fractional Leibniz (see, e.g., Theorem 1 of Grafakos-Oh \cite{Grafakos_Oh_2014}) 
\begin{align*}
\| \<\partial _x\>^{s'}  (fg)  \|_{L^2} \lesssim \|\<\partial _x\>^{s'} f \|_{L^{p_1}} \|  g \|_{L^{{q_1}}} + \|\<\partial _x \>^{s'} g \|_{L^{p_2}} \|  f \|_{L^{{q_2}}}, 
\end{align*} 
where $1/p_j+ 1/q_j =1$ and $1 < p_1,p_2,q_1,q_2 \leq \infty$, and the Sobolev embedding, we have \begin{align*} 
 \left\| I^2 u \right\|_{s'} & \lesssim   \| I^2 \|_{L^{\infty}} \|  u \|_{s'} + \| I^2 \|_{s'} \| u \|_{L^{\infty}}    \lesssim \| I^2 \|_1 \| u \|_1 \lesssim \| u \|_2^5 , \\ 
 \left\| I ( \partial_x u ) \right\|_{s '} & \lesssim  \| I \|_{L^{\infty}} \|  \partial _x u \|_{s'} + \| I \|_{s'} \| \partial _x u \|_{L^{\infty}}   \lesssim \| I \|_1 \| u \|_2 \lesssim \| u \|_2^3  , \\ 
 \left\| (\partial _x I) u  \right\|_{s'} & \lesssim \| (\partial _x I)   \|_{s'} \| u \|_{L^{\infty}} + \| (\partial_x  I)  \|_{} \| \J{\partial _x }^{s'} u \|_{L^{\infty}} 
 \\ & \lesssim  \left( \| (\partial _x^2 u) u \|_{s'} + \| (\partial _x u) ^2 \|_{s'} \right) \| u \|_{L^{\infty}}  +  \| I  \|_1  \| u \|_{1+s'}    
 \\ & \lesssim \left(  \| u \|_s \| u \|_{L^{\infty}} +  \| u \|_2 \| (\partial_x u) \|_{L^{\infty}} + \| u \|_{1+s'} \| (\partial _x u) \|_{L^{\infty}}  \right) \| u \|_1  + \| u \|_2^3\\
 & \lesssim \| u \|_s^3,
\end{align*} 
and hence 
\begin{align*}
\| e^{i |u|^2 } u \|_s \lesssim  (1+\| u \|_{s}^{1+2\lceil s\rceil}), \quad \| e^{i |u|^2 } F(u) \|_s \lesssim  (1+\| u \|_{s}^{3+2\lceil s\rceil}). 
\end{align*}
By taking $u = | \lambda \log t |^{\frac12} \widehat{u_+}   $, we obtain the desired estimates. \end{proof}


\section{The proof of Theorem \ref{theorem_NLS_V_1}}
\label{section_3}
This section is devoted to the proof of Theorem \ref{theorem_NLS_V_1}. 
\subsection{Integral equation}
\label{subsection_integral_NLS_V}
We begin with deriving an appropriate integral equation associated with \eqref{NLS_V} subjected to the asymptotic condition \eqref{theorem_NLS_V_1_1}. To this end, we assume for a while that $u$ is a smooth solution to \eqref{NLS_V}. Recalling the Dollard decomposition $U(t)=\mathcal M(t)\mathcal D(t)\F \mathcal M(t)$ and $M_{\Psi}(t)=e^{i\Psi(t,x)}$, we define the linear modified free propagator associated with \eqref{NLS_V} by
\begin{align}
\label{U_Psi}
U_{\Psi}(t)=\mathcal M_\Psi(t)\mathcal D(t)\F \mathcal M(t).
\end{align}
Let $\chi\in C_0^\infty(\R)$ be  given in the definition of $V_{T_1}$ (see \eqref{V_T_1} above) and $\chi_t(x)=\chi(x/t)$. In what follows, we often omit the variables $t,x$ to write $U_\Psi=U_\Psi(t)$ and so on for short. Since 
$$
U_\Psi\mathcal F^{-1}=\mathcal M_{\Psi}\mathcal D+\mathcal M_{\Psi}\mathcal D\mathcal R,\quad \mathcal R=\mathcal F(\mathcal M-1)\mathcal F^{-1},$$ the asymptotic profile $u_\mathrm p=\mathcal M_\Psi \mathcal Dw_{\mathrm p}$ is decomposed as
\begin{align*}
u_\mathrm p
=(1-\chi_t)\mathcal M_\Psi \mathcal D w_{\mathrm{p}}+\chi_t \mathcal M_\Psi \mathcal D w_{\mathrm{p}}=(1-\chi_t)U_\Psi\mathcal F^{-1}w_{\mathrm{p}}-(1-\chi_t)\mathcal M_\Psi  \mathcal D \mathcal R w_{\mathrm{p}},
\end{align*}
where we have used the support properties $\supp \chi\subset \{|x|\le c_0/2\}$ and $\supp \widehat{u_+}\subset \{|x|\ge c_0\}$ to obtain that $\chi(x) w_{\mathrm p}(t,x)=e^{-i\lambda|\widehat{u_+}(x)|^2\log t}\chi (x)\widehat{u_+}(x)\equiv0$ and hence
$$\chi_t \mathcal M_\Psi(t) \mathcal D(t) w_{\mathrm{p}}(t)=\mathcal M_\Psi(t) \mathcal D(t)\chi w_{\mathrm{p}}(t)\equiv0.$$ Since $w_p$ solves
\begin{align}
\label{w_p}
i\partial_t w_p=t^{-1}F(w_{\mathrm p}),\quad t\neq0,\ x\in \R,
\end{align}
 we have
\begin{align*}
e^{-itH}i\partial_t e^{itH}(1-\chi_t)U_\Psi \mathcal F^{-1}w_{\mathrm{p}}
=t^{-1}(1-\chi_t)U_\Psi \mathcal F^{-1} F(w_{\mathrm{p}})+e^{-itH}[i\partial_t ,e^{itH}(1-\chi_t)U_\Psi \mathcal F^{-1}]w_{\mathrm{p}},
\end{align*}
where $[A,B]=AB-BA$ denotes the commutator. The term $t^{-1}U_\Psi \mathcal F^{-1} F(w_{\mathrm{p}})$ is written as
\begin{align*}
t^{-1}U_\Psi \mathcal F^{-1}F(w_{\mathrm{p}})
&=t^{-1}\mathcal M_\Psi \mathcal D F(\mathcal D^{-1}\mathcal M_\Psi^{-1}v_\mathrm p)+t^{-1}\mathcal M_\Psi \mathcal D\mathcal RF(w_{\mathrm{p}})\\
&=F(v_\mathrm p)+t^{-1}\mathcal M_\Psi \mathcal D\mathcal RF(w_{\mathrm{p}}),
\end{align*}
which, together with the fact $t^{-1}\chi_t \mathcal M_\Psi\mathcal DF(w_{\mathrm{p}})=t^{-1}\mathcal M_\Psi\mathcal D\chi F(w_{\mathrm{p}})$ vanishes identically, implies
\begin{align*}
t^{-1}(1-\chi_t)U_\Psi \mathcal F^{-1} F(w_{\mathrm{p}})
&=t^{-1}(1-\chi_t)\mathcal M_\Psi\mathcal DF(w_{\mathrm{p}})+t^{-1}(1-\chi_t)\mathcal M_\Psi\mathcal D\mathcal RF(w_{\mathrm{p}})\\
&=F(u_\mathrm p)+t^{-1}\mathcal M_\Psi \mathcal D(1-\chi)\mathcal RF(w_{\mathrm{p}}). 
\end{align*}
For short, we set $v=u-u_\mathrm p$, and
\begin{align*}
E_1(t)&=\mathcal M_\Psi(t)\mathcal D(t) (1-\chi)\mathcal R(t) w_{\mathrm{p}}(t)\\
E_2(t)&=-t^{-1}\mathcal M_\Psi(t) \mathcal D(t)(1-\chi)\mathcal R(t)F(w_{\mathrm{p}}(t)),\\
E_3(t)&=-e^{-itH}[i\partial_t,e^{itH}(1-\chi_t)U_\Psi(t) \mathcal F^{-1}] w_{\mathrm{p}}(t).
\end{align*}
It follows from the above computations, the NLS \eqref{NLS_V} and Lemma \ref{lemma_F} that
$$v-E_1=u-u_{\mathrm p}-E_1=u-(1-\chi_t)U_\Psi \mathcal F^{-1}w_{\mathrm{p}}$$ 
and that
\begin{align}
\nonumber
(i\partial_t -H)(v-E_1)
&= F(u)-e^{-itH}i\partial_t e^{itH} (1-\chi_t)U_\Psi \mathcal F^{-1}w_{\mathrm{p}}\\
\nonumber
&=F(u)-F(u_\mathrm p)+E_2 +E_3\\
\nonumber
&=2\lambda |u_{\mathrm p}|^2v+\lambda u_{\mathrm p}^2\overline{v}+ G(v,u_{\mathrm p})+E_2+E_3. 
\end{align}
This equation with the asymptotic condition $\|v\|\to 0$ as $t\to \infty$ leads the integral equation
\begin{align}
\label{integral_equation}
v(t)=E_1(t)+i\int_t^\infty e^{-i(t-s)H}\left(2\lambda |u_{\mathrm p}|^2v+\lambda u_{\mathrm p}^2\overline{v}+ G(v,u_{\mathrm p})+E_2+E_3\right)\!(s)ds. 
\end{align}

\subsection{Energy estimates}
Given $\delta,b,R>0,T>2$, we define a complete metric space $X$ by
\begin{align*}
X&=X(\delta,b,T,R):=\{f\in C([T,\infty);H^1(\R))\ |\ \|f\|_X\le R\},\\
\|f\|_X&:=\sup_{t\ge T}t^{\delta}(\log t)^{-b}\|f\|_{1},\quad d_X(f,g)=\|f-g\|_X. 
\end{align*}
We let $\Phi[v](t)$ be the RHS of \eqref{integral_equation} and shall show that $v\mapsto \Phi[v]$ is a contraction on $X$ for sufficiently large $T$. Since $D(H)=H^2(\R)$, $e^{-itH}$ leaves $H^1(\R)$ invariant, satisfying
\begin{align} \label{K8/1-1}
\|e^{-itH}f\|_{1}\lesssim \|\<H\>^{1/2}e^{-itH}f\|=\|\<H\>^{1/2}f\|\lesssim \|f\|_{1}
\end{align}
with implicit constants independent of $t$. Hence
\begin{align}
\label{theorem_NLS_V_1_proof_1}
\|\Phi[v](t)\|_1\lesssim \|E_1(t)\|_1+\int_t^\infty \left(\||u_{\mathrm p}|^2v\|_1+\|u_{\mathrm p}^2\overline{v}\|_1+\|G(v,u_{\mathrm p})\|_1+\|E_2\|_1+\|E_3\|_1\right)(s)ds. 
\end{align}
We collect necessary estimates for each terms of this inequality in the following lemmas: 

\begin{lemma}
\label{lemma_NLS_V_1}
For $t>0$, we have
\begin{align*}
\partial_x \mathcal M_{\Psi}(t)&=\mathcal M_\Psi(t)(it^{-1}x+\partial_x)+r_1(t),\\
\partial_x \mathcal M_{\Psi}(t)\mathcal D(t)&=\mathcal M_\Psi(t)\mathcal D(t)(ix+t^{-1}\partial_x)+ {r_1(t) }\mathcal D(t),\\
\partial_xU_{\Psi} (t)\mathcal F^{-1}&=U_{\Psi}(t) \mathcal F^{-1}ix+ {r_2(t) }U_{\Psi} (t)\mathcal F^{-1}. 
\end{align*}
with some multiplication operators $r_j(t)$ by $r_j(t,x)$ satisfying $|r_j(t,x)|\lesssim t^{-\rho_{\mathrm L}'}$. 
\end{lemma}

\begin{proof}
The lemma follows from Proposition \ref{proposition_HJ} and the following three formulas
\begin{align*}
\partial_x \mathcal M_{\Psi}&=\mathcal M_\Psi(it^{-1}x+\partial_x)+i\mathcal M_\Psi(\partial_x\Psi-t^{-1}x),\\
(it^{-1}x+\partial_x)\mathcal D&=\mathcal D(ix+t^{-1}\partial_x),\\
\partial_xU_{\Psi} \mathcal F^{-1}
&=\partial_x \mathcal M_\Psi \mathcal M^{-1}U_0\mathcal F^{-1}\\
&=\mathcal M_\Psi \mathcal M^{-1}\left\{\partial_x+i(\partial_x\Psi-t^{-1}x)\right\}U_0\mathcal F^{-1}\\
&=U_{\Psi} \mathcal F^{-1}ix+i(\partial_x\Psi-t^{-1}x)U_{\Psi} \mathcal F^{-1}. 
\end{align*}
\end{proof}

\begin{lemma}
\label{lemma_NLS_V_2}
Let $u_+\in H^{0,1}\cap H^{1,\sigma}$ with some $\sigma>1/2$. Then, for $t\ge2$,  
\begin{align}
\label{lemma_NLS_V_2_1}
\||u_{\mathrm p}(t)|^2v (t)\|_1+\|u_{\mathrm p}(t)^2\overline{v (t)}\|_1&\le C t^{-1}\|\widehat{u_+}\|_{L^\infty}\|v(t) \|_1,\\
\label{lemma_NLS_V_2_2}
\|G(v (t),u_{\mathrm p}(t))\|_1&\le C\left( t^{-1/2}\|v (t)\|_1^2+\|v (t)\|_1^3\right),
\end{align}
with some constant $C=C(\|\<x\>\widehat{u_+}\|_\sigma,\|\widehat{u_+}\|_1)>0$. 
\end{lemma}

\begin{proof}
For the first estimate \eqref{lemma_NLS_V_2_1}, it is enough to deal with $u_{\mathrm p}^2\overline{v }$, the proof for $|u_{\mathrm p}|^2v $ being even simpler. 
Recall that $u_{\mathrm p}=\mathcal M_{\Psi}\mathcal D w_{\mathrm p}$. 
Since $\|u_{\mathrm p}\|_{L^\infty}=\|\mathcal D(t)w_{\mathrm p}\|_{L^\infty}\le t^{-1/2}\|\widehat{u_+}\|_{L^\infty}$, we have
$$
\|u_{\mathrm p}^2\overline{v }\|\le \|u_{\mathrm p}\|_{L^\infty}^2\|v \|\lesssim t^{-1}\|\widehat{u_+}\|_{L^\infty}^2\|v \|.
$$
{We next calculate by using the  previous lemma that
$$
\partial_x (u_{\mathrm p}^2\overline{v })
=2u_{\mathrm p}(\partial_xu_{\mathrm p})\overline{v }+u_{\mathrm p}^2\overline{\partial_x v }=2u_{\mathrm p}\overline{v }\left\{\mathcal M_{\Psi}\mathcal D(ix+t^{-1}\partial_x)w_{\mathrm p}+r_1\mathcal D w_{\mathrm p}\right\}+u_{\mathrm p}^2\overline{\partial_x v },
$$
where $\partial_xw_{\mathrm p}$ is of the form
$$
\partial_xw_{\mathrm p}=-2i\lambda (\log t)\Re(\widehat u_+\partial_x\widehat{u_+})w_{\mathrm p}+e^{-i\lambda(\log t)|\widehat{u_+}|^2}\partial_x\widehat{u_+}
$$
from which we know
\begin{align*}
\|\mathcal D\partial_xw_{\mathrm p}\|=\|\partial_xw_{\mathrm p}\|&\lesssim \log t(1+\|\widehat{u_+}\|_{L^\infty}^2)\|\partial_x\widehat{u_+}\|. 
\end{align*}
Therefore
\begin{align*}
\|\partial_x (u_{\mathrm p}^2\overline{v })\|
&\lesssim \|u_{\mathrm p}\|_{L^\infty}\|v \|\|\mathcal Dxw_{\mathrm p}\|_{L^\infty}+t^{-1}\|u_{\mathrm p}\|_{L^\infty}\|v \|_{L^\infty}\|\mathcal D\partial_xw_{\mathrm p}\|\\
&\quad+t^{-\rho_{\mathrm L}'}\|u_{\mathrm p}\|_{L^\infty}\|v \|\|\mathcal Dw_{\mathrm p}\|_{L^\infty}+\|u_{\mathrm p}\|_{L^\infty}^2\|v \|_1\\
&\lesssim t^{-3/2}\|\widehat{u_+}\|_{L^\infty}\|x\widehat{u_+}\|_{L^\infty}\|v \|+t^{-3/2}\log t\|\widehat{u_+}\|_{L^\infty}(1+\|\widehat{u_+}\|_{L^\infty}^2)\|\partial_x \widehat{u_+}\| \|v \|_1\\
&\quad +t^{-1-\rho_{\mathrm L}'}\|\widehat{u_+}\|_{L^\infty}^2\|v \|+t^{-1}\|\widehat{u_+}\|_{L^\infty}^2\|v \|_1\\
&\le C t^{-1}\|\widehat{u_+}\|_{L^\infty}\|v \|_1
\end{align*}
with some $C=C(\|\<x\>\widehat{u_+}\|_\sigma,\|\widehat{u_+}\|_1)>0$, and \eqref{lemma_NLS_V_2_1} follows. For \eqref{lemma_NLS_V_2_2},  recalling that $$G(v ,u_{\mathrm p})=2\lambda\Re[v u_{\mathrm p}]v +\lambda|v |^2u_{\mathrm p}+\lambda|v |^2v $$
we similarly obtain
\begin{align*}
\|G(v ,u_{\mathrm p})\|&\le 3|\lambda| \|u_{\mathrm p}\|_{L^\infty}\|v \|_{L^\infty}\|v \|+|\lambda|\|v \|_{L^\infty}^2\|v \|\lesssim t^{-1/2}\|\widehat{u_+}\|_{L^\infty} \|v \|_1^2+\|v \|_1^3.
\end{align*}
Since 
$$
\partial_x (u_{\mathrm p}v ^2)=v ^2\left\{\mathcal M_\Psi\mathcal D(ix+t^{-1}\partial_x)w_{\mathrm p}+r_1\mathcal Dw_{\mathrm p}\right\}+2u_{\mathrm p}v \partial_x v, 
$$
a similar argument as above also shows
\begin{align*}
\|\partial_x (u_{\mathrm p}v ^2)
&\|\lesssim \|v \|_{L^\infty}\|v \|\|\mathcal Dxw_{\mathrm p}\|_{L^\infty}+t^{-1}\|v \|_{L^\infty}^2\|\mathcal D\partial_xw_{\mathrm p}\|+t^{-\rho_{\mathrm L}'}\|v \|_{L^\infty}\|v \|\|\mathcal Dw_{\mathrm p}\|_{L^\infty}\\
&\quad +\|u_{\mathrm p}\|_{L^\infty}\|v \|_{L^\infty}\|\partial_xv \|\\
&\le C t^{-1/2}\|v \|_1^2
\end{align*}
and similarly 
$$
\|\partial_x (u_{\mathrm p}|v |^2)\|\le C t^{-1/2}\|v \|_1^2
$$
with some $C=C(\|\<x\>\widehat{u_+}\|_\sigma,\|\widehat{u_+}\|_1)>0$.  Hence
\begin{align*}
\|\partial_x G(v ,u_{\mathrm p})\|
&\lesssim  t^{-1/2}\|v \|_1^2+\|v \|_1^3
\end{align*}
and \eqref{lemma_NLS_V_2_2} follows.}
\end{proof}

\begin{lemma}
\label{lemma_NLS_V_3}
Suppose $0\le \delta\le1$ and $u_+\in H^{1,2\delta}\cap H^{0,1}$. Then, for $t>2$, 
\begin{align*}
\|E_1(t)\|_1+t\| E_2(t)\|_1&\lesssim t^{-\min\{\delta,\rho_{\mathrm L}'\}}(\log t)^{2}.
\end{align*}
\end{lemma}

\begin{proof}
By Lemmas \ref{lemma_R} and \ref{lemma_nonlinear} and the unitarity of $\mathcal M_\Psi(t)\mathcal D(t)$, 
\begin{align*}
\|E_1\|+t\|E_2\|\le C(\|\widehat{u_+}\|_{2\delta})t^{-\delta}(\log t)^{2}. 
\end{align*}
It follows from Lemmas \ref{lemma_R},  \ref{lemma_nonlinear} and \ref{lemma_NLS_V_1} that
\begin{align*}
\|\partial_xE_1\|&\lesssim \|(ix+t^{-1}\partial_x)(1-\chi)\mathcal R w_{\mathrm p}\|+t^{-\rho_{\mathrm L}'}\|\mathcal R w_{\mathrm p}\|\\
&\lesssim \|x\mathcal R w_{\mathrm p}\|+t^{-1}\|\mathcal R w_{\mathrm p}\|_1+(t^{-\rho_{\mathrm L}'}+t^{-1})\|\mathcal Rw_{\mathrm p}\|\\
&\lesssim t^{-\delta}\|xw_{\mathrm p}\|_{2\delta}+t^{-1}\|w_{\mathrm p}\|_1+t^{-\rho_{\mathrm L}'}\|w_{\mathrm p}\|\\
&\le C t^{-\min\{\delta,\rho_{\mathrm L}'\}}(\log t)^{2}
\end{align*}
with some $C=C(\|\<x\>\widehat{u_+}\|_{2\delta},\|\widehat{u_+}\|_1)>0$. Similarly,
\begin{align*}
t\|\partial_x E_2(t)\|
&\lesssim \|(ix+t^{-1}\partial_x)(1-\chi)\mathcal R F(w_{\mathrm p})\|+t^{-\rho_{\mathrm L}'}\|\mathcal R F(w_{\mathrm p})\|\\
&\le Ct^{-\min\{\delta,\rho_{\mathrm L}'\}}(\log t)^{2}
\end{align*}
and the desired bound follows.  
\end{proof}

\begin{lemma}
\label{lemma_NLS_V_4}
 There exists $T_2$ depending on $ V^{\mathrm C}$ such that, for any $u_+\in H^{1,1}(\R)$, 
  $$
\|E_3(t)\|_1\lesssim t^{-\min\{1+\rho_{\mathrm L}',\rho_{\mathrm S}\}}\log t,\quad t>T_2. 
$$
\end{lemma}

\begin{proof}Recalling that $V^{\mathrm C}$ is compactly supported and $\chi_t(x)=1$ if  $|x|\le c_0t/4$, we take $T_2$ so large that $(1-\chi_t)V^{\mathrm C}\equiv0$ for $t>T_2$. 
It follows from direct calculations (see \cite[Lemma 4.4]{Kawamoto_Mizutani_TAMS} and its proof) that $E_3$ is decomposed into four parts as $E_3=I_1+I_2+I_3+I_4$, where 
\begin{align*}
I_1&=-i(1-\chi_t)\mathcal M_{\Psi}(t)\mathcal A_{\Psi}(t)\mathcal D(t)\mathcal F\mathcal M(t)\mathcal F^{-1}w_{\mathrm p},\\
I_2&=-\left\{\frac it\left(\partial_x\Psi-\frac xt\right)(\chi')_t+\frac{1}{2t^2}(\chi'')_t\right\}U_{\Psi}(t)\mathcal F^{-1}w_{\mathrm p},\\
I_3&=-\frac{1}{t^2}(\chi')_t\partial_x U_{\Psi}(t)\mathcal F^{-1}w_{\mathrm p},\\
I_4&=(1-\chi_t)(V^{\mathrm S}+V^{\mathrm C})U_{\Psi}(t)\mathcal F^{-1}w_{\mathrm p}=(1-\chi_t)V^{\mathrm S}U_{\Psi}(t)\mathcal F^{-1}w_{\mathrm p},
\end{align*}
where $(\chi')_t(x)=\chi'(x/t)$, $(\chi'')_t(x)=\chi''(x/t)$ and 
\begin{align*}
\mathcal A_{\Psi}(t)=\left(\partial_x\Psi-\frac xt\right)\partial_x+ { \frac12 \left( \partial_x^2\Psi-\frac{1}{t} \right) } .
\end{align*}
{For short, we set $a_1=\partial_x\Psi-t^{-1}x$ and $a_2=(\partial_x^2\Psi-t^{-1})/2$. Then $I_1$ is of the form
\begin{align*}
I_1=-i(1-\chi_t)\mathcal M_\Psi(a_1\partial_x+a_2)\mathcal D\mathcal F\mathcal M\mathcal F^{-1}{w_{\mathrm p}}=-i(1-\chi_t)(a_1U_\Psi\mathcal F^{-1}t^{-1}\partial_x+a_2U_\Psi\mathcal F^{-1})w_{\mathrm p}.
\end{align*}
Since $U_\Psi \mathcal F^{-1}$ is unitary on $L^2(\R)$, we obtain by using Proposition \ref{proposition_HJ} that
\begin{align*}
\|I_1\|\lesssim t^{-1-\rho_{\mathrm L}'}(\|\partial_x w_{\mathrm p}\|+\|w_{\mathrm p}\|)\le C(\|\widehat{u_+}\|_1) t^{-1-\rho_{\mathrm L}'}\log t.
\end{align*}
To estimate $\partial_x I_1$, we use Lemma \ref{lemma_NLS_V_1} to calculate
\begin{align*}
i \partial_x I_1&=-t^{-1}\chi'_t(a_1U_\Psi\mathcal F^{-1}t^{-1}\partial_x+a_2U_\Psi\mathcal F^{-1})w_{\mathrm p}\\
&\quad +(1-\chi_t)\{(\partial_xa_1)+a_1r_2\}U_\Psi\mathcal F^{-1}t^{-1}\partial_xw_{\mathrm p}+\{(\partial_x a_2)+a_2r_2\}U_\Psi\mathcal F^{-1}w_{\mathrm p}\\
&\quad +a_1 U_\Psi\mathcal F^{-1}ixt^{-1}\partial_xw_{\mathrm p}+a_2U_\Psi\mathcal F^{-1}ixw_{\mathrm p},
\end{align*}
which, combined with Lemma \ref{lemma_nonlinear} and Proposition \ref{proposition_HJ}, implies
\begin{align*}
\|\partial_x I_1\|
&\lesssim t^{-2-\rho_{\mathrm L}'}\|w_{\mathrm p}\|+t^{-\min\{2+\rho_{\mathrm{L}}',1+2\rho_{\mathrm L}'\}}\|\partial_xw_{\mathrm p}\|+t^{-1-\rho_{\mathrm L}'}(\|w_{\mathrm p}\|+\|x\partial_xw_{\mathrm p}\|+\|xw_{\mathrm p}\|)
\\
&\le C(\|\<x\>\widehat{u_+}\|_1)t^{-1-\rho_{\mathrm L}'}\log t.
\end{align*}
}
Hence,
\begin{align}
\label{lemma_NLS_V_4_1}
\|I_1\|_1\le C(\|\<x\>\widehat{u_+}\|_1) t^{-1-\rho_{\mathrm L}'}\log t. 
\end{align}
Similarly, Lemmas \ref{lemma_nonlinear} and \ref{lemma_NLS_V_1} and Proposition \ref{proposition_HJ} yield
\begin{align}
\label{lemma_NLS_V_4_2}
\|I_2\|_1&\lesssim (t^{-1-\rho_{\mathrm L}'}+t^{-2})\|\<x\>w_{\mathrm p}\|\le C(\|\<x\>\widehat{u_+}\|)  t^{-1-\rho_{\mathrm L}'},\\
\label{lemma_NLS_V_4_3}
\|I_3\|_1&\lesssim t^{-2}\|\<x\>w_{\mathrm p}\|\le C(\|\<x\>\widehat{u_+}\|)  t^{-2},
\end{align}
and, with Assumption \ref{assumption_A}, 
\begin{align}
\nonumber
\|I_4\|_1&\lesssim \left(\|(1-\chi_t)V^{\mathrm S}\|_{L^\infty}+\|(1-\chi_t)\partial_x V^{\mathrm S}\|_{L^\infty}+t^{-1}\|(\chi')_tV^{\mathrm S}\|_{L^\infty}\right)\|\<x\>w_{\mathrm p}\|\\
\label{lemma_NLS_V_4_4}
&\le C(\|\<x\>\widehat{u_+}\|)t^{-\rho_{\mathrm S}}.
\end{align}
The desired bound now follows from \eqref{lemma_NLS_V_4_1}--\eqref{lemma_NLS_V_4_4}. 
\end{proof}

We are now in a position to conclude the proof of Theorem \ref{theorem_NLS_V_1}. 
\begin{proof}[Proof of Theorem \ref{theorem_NLS_V_1}]
Recall that $\|f\|_X=\sup_{t\ge T}t^{\delta}(\log t)^{-b}\|f\|_{1}$. Let $\delta,b$ be as in Theorem \ref{theorem_NLS_V_1}, $\delta<\rho_{\mathrm L}'<\rho_{\mathrm L}$, $T>\max\{T_1,T_2\}$ and $v\in X$. Then 
$$
\|v(t)\|_1\le t^{-\delta}(\log t)^b R
$$
 for $t\ge T$. It thus follows from lemmas \ref{lemma_NLS_V_2}--\ref{lemma_NLS_V_4} and \eqref{theorem_NLS_V_1_proof_1} that \begin{align*}
\|\Phi[v ](t)\|_1
&\lesssim t^{-\delta}(\log t)^{2}+\int_t^\infty \Big(s^{-1-\delta}(\log s)^b\|\widehat{u_+}\|_{L^\infty}R+s^{-1/2-2\delta}(\log s)^{2b}R^2\\&\quad\quad \quad\quad \quad\quad \quad\quad +s^{-3\delta}(\log s)^{3b}R^3+s^{-1-\delta_0}(\log s)^{2}+s^{-1-\delta_0}\Big)ds\\
&\lesssim t^{-\delta}(\log t)^{2}+t^{-\delta}(\log t)^b\|\widehat{u_+}\|_{L^\infty}R+t^{1/2-2\delta}(\log t)^{2b}R^2+t^{1-3\delta}(\log t)^{3b}R^3.
\end{align*}
This estimate implies
\begin{align}
\label{proof_theorem_1}
\|\Phi[v ]\|_X\lesssim (\log T)^{2-b}+\|\widehat{u_+}\|_{L^\infty}R+T^{1/2-\delta}(\log T)^{b}R^2+T^{1-2\delta}(\log T)^{2b}R^3.
\end{align}
Similarly, if $v_1,v_2\in X$, then 
$$
\|v_j(t)\|_1\le t^{-\delta}(\log t)^b R,\quad \|v_1(t)-v_2(t)\|_1\le t^{-\delta}(\log t)^b\|v_1-v_2\|_X,
$$
for $t\ge T$ and $j=1,2$. 
Since $G(v_{1},u_{\mathrm p})-G(v_{2},u_{\mathrm p})$ is of the form
\begin{align*}
2\lambda (\Re(v_{1}u_{\mathrm p})v_{1}-\Re(v_{2}u_{\mathrm p})v_{2})+\lambda (|v_1|^2-|v_2|^2)u_{\mathrm p}+\lambda (|v_{1}|^2v_{1}-|v_{2}|^2v_{2}),
\end{align*}
we have
\begin{align*}
&\|G(v_{1}(t),u_{\mathrm p}(t))-G(v_{2}(t),u_{\mathrm p}(t))\|_1\\
&\lesssim \|u_{\mathrm{p}}(t)\|_{L^\infty}\left(\|v_{1}(t)\|_1+\|v_{2}(t)\|_1\right)\|v_{1}(t)-v_{2}(t)\|_1+\left(\|v_{1}(t)\|_1^2+\|v_{2}(t)\|_1^2\right)\|v_{1}(t)-v_{2}(t)\|_1\\
&\lesssim \left(t^{-1/2-2\delta}(\log t)^{2b}R+t^{-3b}(\log t)^{3b}R^2\right)\|v_{1}-v_{2}\|_X
\end{align*}
Plugging this bound into the difference $\Phi[v_{1}](t)-\Phi_1[v_{2}](t)$, which is of the form
$$
i\int_t^\infty e^{-i(t-s)H}\left(2\lambda |u_{\mathrm p}|^2(v_1-v_2)+\lambda u_{\mathrm p}^2\overline{(v_1-v_2)}+ G(v_1,u_{\mathrm p})-G(v_2,u_{\mathrm p})\right)(s)ds,
$$
implies
\begin{align*}
&\|\Phi[v_1](t)-\Phi[v_2](t)\|_1\\
&\lesssim \int_t^\infty \left(s^{-1-\delta}(\log s)^b\|\widehat{u_+}\|_{L^\infty}+s^{-1/2-2\delta}(\log s)^{2b}R+s^{-3b}(\log s)^{3b}R^2\right)\|v_{1}-v_{2}\|_Xds\\
&\lesssim \left(\|\widehat{u_+}\|_{L^\infty}t^{-1}(\log t)^b+t^{1/2-2\delta}(\log t)^{2b}R+t^{1-3\delta}(\log T)^{3b}R^2\right)\|v_1-v_2\|_X 
\end{align*}
and hence
\begin{align}
\label{proof_theorem_2}
\|\Phi[v_1]-\Phi[v_2]\|_X\lesssim \left(\|\widehat{u_+}\|_{L^\infty}+T^{1/2-\delta}(\log T)^{b}R+T^{1-2\delta}(\log T)^{2b}R^2\right)\|v_1-v_2\|_X. 
\end{align}
It follows from \eqref{proof_theorem_1} and \eqref{proof_theorem_2} that for any $R>0$, there exist $\ep>0$ and $T_0>\max\{T_1,T_2\}$ such that if  $\|\widehat{u_+}\|_{L^\infty}<\ep$ and $T>T_0$, then $\Phi[v]$ is a contraction on $X(\delta,b,T,R)$. Therefore, we obtain a unique solution $u\in C([T,\infty);H^1(\R))$ to \eqref{integral_equation}. 

Next, by the completely same argument as that in the proof of \cite[Lemma 2.1]{Kawamoto_Mizutani_TAMS}, we find that $u$ satisfies the usual Duhamel formula: 
\begin{align}
\label{Duhamel}
u(t)=e^{-i(t-T)H}u(T)-i\int_T^t e^{-i(t-s)H}F(u(s))ds,\quad t\ge T. 
\end{align}
Hence, $u$ is (by definition) the $H^1$-solution to \eqref{NLS_V} subjected to the condition \eqref{theorem_NLS_V_1_1}. 

Finally, it is well-known that the Cauchy problem for \eqref{NLS_V} is globally well-posed in $H^1(\R)$. Indeed, since Assumption \ref{assumption_A} implies $V=V^{\mathrm S}+V^{\mathrm L}+V^{\mathrm C}$ with $V^{\mathrm S}+V^{\mathrm L}\in L^\infty(\R)$ and $V^{\mathrm C}\in L^p(\R)$ for all $1\le p\le 2$, $e^{-itH}$ satisfies the local-in-time Strichartz estimates
$$
\|e^{-itH}f\|_{L^p([-t_0,t_0];L^q(\R))}\le C\|f\|,\quad t_0>0,\ 1/p=1/2-1/q,\ 2\le p,q\le \infty,
$$
with some $C$ depending on $t_0$ (see Yajima \cite[Corollary 1.2]{Yaj}). Moreover, the cubic nonlinearity $|u|^2u$ is energy subcritical in one space dimension. Therefore, the standard argument for the case without potentials (see e.g. Ginibre--Velo \cite{Ginibre_Velo_1978}) also works well for  \eqref{NLS_V}. Namely, by using the above local-in-time Strichartz estimates, one can show the local well-posedness for  \eqref{NLS_V} in $H^1(\R)$ with the maximal existence time depending only on the $H^1$-norm of the initial datum $u(T)$. Moreover, the mass $\|u\|^2$ and the energy $E(u)=\|\nabla u\|^2/2+\<Vu,u\>+\lambda\|u\|_{L^4}^4/4$ are conserved by the flow of \eqref{NLS_V} and satisfies $E(u(t))+C\|u(t)\|^2\sim \|u(T)\|^2_1$ with some $C>0$ independent of $t$ which, together with the local well-posedness, leads the global well-posedness in $H^1(\R)$. As a consequence, one can extend the above solution $u\in C([T,\infty);H^1(\R))$ backward in time by solving the Cauchy problem with the initial datum $u(T)$, obtaining $u\in C(\R;H^1(\R))$. 
\end{proof}

\appendix

\section{The proof of Proposition \ref{proposition_HJ}}
\label{appendix_proof_proposition_HJ}

Here we prove Proposition \ref{proposition_HJ}. As already mentioned, the proof is almost identical to that of  \cite[Proposition 3.1]{Kawamoto_Mizutani_TAMS}. Hence we give an outline of the proof only,  and refer to \cite[Proposition 3.1]{Kawamoto_Mizutani_TAMS} for more details. The proof is based on the method bicharacteristic. As a first step, we consider the Hamilton equations
\begin{align}
\label{A_1}
\partial_t Z(t,\xi)=\Xi(t,\xi),\quad \partial_t \Xi(t,\xi)=-(\partial_x V_{T_1})(t,Z(t,\xi)),\quad t\ge0,
\end{align}
with the condition 
\begin{align}
\label{A_2}
Z(0,\xi)=0,\quad \ds\lim_{t\to \infty}\Xi(t,\xi)=\xi. 
\end{align}
Given $\xi\in \R$, there exists a unique solution $(Z(t),\Xi(t))$ to \eqref{A_1}--\eqref{A_2} satisfying 
\begin{align}
\label{A_3}
|\partial_\xi^k (Z(t, \xi)  - t\xi) |\lesssim \theta(t),
\quad
|\partial_\xi^k(\Xi(t, \xi) -\xi)|\lesssim (t+T_1)^{-\rho_{\mathrm L}},
\end{align}
for $k=0,1,2$, where 
$$
\theta(t)=\begin{cases}
\J{t} (t+T_1)^{-\rho_{\mathrm L}}&\text{if}\ \rho_{\mathrm L}<1,\\ \log(tT_1^{-1}+2)&\text{if}\ \rho_{\mathrm L}=1.
\end{cases}
$$
To prove this, we observe that the problem \eqref{A_1}--\eqref{A_2} is equivalent to 
\begin{equation}
\label{A_4}
\left\{
\begin{aligned}
Z(t,\xi)&=
t\xi+\int_0^\infty \min\{t,s\}(\partial _x V_{T_1})(s,Z(s,\xi))ds,\\
\Xi(t,\xi)&=\xi+\int_t^\infty (\partial _x V_{T_1})(s,Z(s,\xi))ds. 
\end{aligned}
\right.
\end{equation}
Then the solution $(Z,\Xi)$ can be constructed by showing that the map from $Z(t)$ to the RHS of the first equation of \eqref{A_4} is a contraction on the complete metric space
\begin{align*}
\mathcal Z&=\{Z\in C^1([0,\infty);C^2(\R))\ |\ \|Z\|_{\mathcal Z}\le M\},\\
\|Z\|_{\mathcal Z}&=\sup_{\xi\in \R}\sup_{t\ge 0}\frac{1}{\theta(t)}\sum_{0\le k\le 2}|\partial_\xi^k(Z(t)-t\xi)|,\quad
d_{\mathcal Z}(Z_1,Z_2)=\|Z_1-Z_2\|_{\mathcal Z}
\end{align*}
with some $M>0$ independent of $T_1$ and $\xi$. Next, for sufficiently large $T_1>$ and all $t\ge0$, we let $\xi\mapsto \eta(t,\xi)$ be the inverse of $\xi\mapsto \Xi(t,\xi)$ and define $S(t,\xi):=\varphi(t,\eta(t,\xi))$ where
\begin{align*}
\varphi(t, \xi) = \int_0^t\left(\frac12|\Xi(\tau,\xi)|^2+V_{T_1}(\tau,Z(\tau,\xi))+Z(\tau,\xi)\cdot (\partial_t \Xi)(\tau,\xi)\right)d\tau.
\end{align*} 
By calculating $\partial_\xi S(t,\xi)$ and $\partial_t[S(t,\Xi(t,\xi))]$, and using \eqref{A_1}, one finds that $S(t,\xi)$ satisfies
\begin{equation}
\left\{\nonumber
\begin{aligned}
&\partial _t S(t,\xi) = \frac12 |\xi| ^2 +V_{T_1}(t, \nabla _{\xi} S(t,\xi)),\\ 
&\partial_\xi S(t,\xi) =Z(t,\eta(t,\xi)),\\
&|\partial_\xi^k (\partial_\xi S(t,\xi) - t\xi)|\lesssim\widetilde\theta(t),\quad k=0,1,2. 
\end{aligned}
\right.
\end{equation}
where 
$$
\widetilde \theta(t)=\begin{cases}
t(t+T_1)^{-\rho_{\mathrm L}}&\text{if}\ \rho_{\mathrm L}<1,\\ \log(tT_1^{-1}+1)&\text{if}\ \rho_{\mathrm L}=1. 
\end{cases}
$$
Then the map $\R\ni \xi\mapsto t^{-1}\partial_\xi S(t,\xi)$ is diffeomorphic and its inverse $\widetilde \Theta(t,\xi)$ satisfies $$\partial_\xi S(t,\widetilde \Theta(t,\xi))=t\xi.$$ Define, for $x\in \R$ and $t>0$,  
$$
\Theta(t,x)=\widetilde \Theta(t,x/t),\quad \Psi(t,x)=x\Theta(t,x)-S(t,\Theta(t,x)). 
$$
It follows by plugging $\xi=\Theta(t,x)$ into the above estimates for $S(t,\xi)$ that
\begin{align}
\label{A_5}
|\partial_x^k(\Theta(t,x)-t^{-1}x)|\lesssim t^{-k-1}\widetilde\theta(t),\quad t>0,\quad x\in \R, 
\end{align}
for $k=0,1,2$. Moreover, $\Psi$ satisfies $\partial_x\Psi(t,x)=\Theta(t,x)$ and hence
$$
\partial_t \Psi=-(\partial_t S)(t,\Theta)=\frac12|\Theta|^2-V_{T_1}(t,(\partial_xS)(t,\Theta))=-\frac12|\partial_x \Psi|^2-V_{T_1}(t,x).
$$
Finally \eqref{proposition_HJ_1} follows from \eqref{A_5}. This completes the proof of Proposition \ref{proposition_HJ}.

\section*{Acknowledgments}
M. K. is partially supported by JSPS KAKENHI Grant Number JP24K06796. H. M. is partially supported by JSPS KAKENHI Grant Numbers JP21K03325 and JP24K00529. 


\end{document}